\documentclass[11pt]{article}
\usepackage{array}
\usepackage{booktabs}
\usepackage{mathdots}
\usepackage{epsfig,subfigure}
\usepackage{color}
\usepackage{mathrsfs,amsbsy}
\usepackage{dsfont}
\usepackage{hyperref}
\usepackage[all]{xy}
\usepackage{float}
\usepackage{graphicx}
\usepackage{amsmath,amsfonts,amsthm}
\usepackage{amssymb,latexsym}
\usepackage{multirow}
\usepackage{todonotes}
\usepackage{authblk}
\allowdisplaybreaks

\newtheorem{theorem}{Theorem}[section]

\newtheorem{lemma}{Lemma}[section]

\newtheorem{remark}{{\sc Remark}}[section]
\newtheorem{example}{Example}[section]

\newcommand\blfootnote[1]{
  \begingroup
  \renewcommand\thefootnote{}\footnote{#1}
  \addtocounter{footnote}{-1}
  \endgroup
}

\DeclareMathOperator{\diag}{diag}
\DeclareMathOperator{\rand}{rand}

\title{Perturbation analysis of the extinction probability of\\ a Markovian binary tree
\thanks{This research was supported in part by NSFC under grant 11301013.} }
\author[1]{Pei-chang Guo}
\author[1]{Yun-feng Cai}
\author[2]{Jiang Qian}
\author[1]{Shu-fang Xu}
\affil[1]{LMAM \& School of Mathematical Sciences, Peking University, Beijing 100871, China}
\affil[2]{School of Sciences, Beijing University of Posts and Telecommunications, Beijing 100876, China}

\date{ }

\begin{document}
\maketitle
\blfootnote{Email addresses: gpeichang@126.com (P.C. Guo)}

\vspace{-14mm}
\begin{abstract}
The extinction probability of the Markovian Binary Tree (MBT) is the minimal nonnegative solution of a Quadratic Vector Equation (QVE).
In this paper,  we present a perturbation analysis for the extinction probability of a supercritical MBT.
We derive a perturbation bound for the minimal nonnegative solution of the QVE,
and an error bound is also given, which can be used to measure the quality of an approximation solution.
Numerical tests show that these bounds are fairly sharp.

\vspace{2mm} \noindent \textbf{Keywords.} perturbation analysis, Markovian binary tree, quadratic vector equation, minimal nonnegative solution

%

\end{abstract}

\section{Introduction}\label{sec:intro}
We first introduce some necessary notations for this paper.
For matrices $A=[a_{ij}], B=[b_{ij}]\in \mathbb{R}^{m\times n}$, we write $A\geq B (A>B)$
if $a_{ij}\ge b_{ij} (a_{ij}>b_{ij})$ holds for all $i, j$.
For vectors $x,y \in \mathbb{R}^n$,
we write $x\geq y (x>y)$ if $x_i\geq y_i (x_i>y_i)$ holds for all $i=1, 2,\dots,n$.
For square matrix $A$, its spectral radius is denoted by $\rho(A)$.
The symbol $\|\cdot\|$ is used to denote the infinity-norm of the dumb matrix/vector unless otherwise specified.
The column vector of all ones is denoted by $e$, i.e., $e=\begin{bmatrix}1 & 1 & \dots & 1\end{bmatrix}^{\top}$.

In this paper, we consider the perturbation analysis of the extinction probability  of a Markovian Binary Tree (MBT).
MBTs belong to a special class of continuous-time Markovian multi-type branching processes \cite{athreya},
which are used to model the growth for populations consisting of several types of individuals who may reproduce and die during their lifetime.
Applications have been found  in biology and epidemiology \cite{appa,appb},
and telecommunication systems \cite{appc,appd}.
We refer the readers to \cite{deta,detb, newmod,detc} for a detailed description of MBTs.

For MBTs, the individuals give birth to only one child at a time and the life of each individual is controlled by a Markovian process,
called the phase process, on the state space $\{1,2,\dots,n\}$.
A fundamental problem is what the extinction probability is for all individuals will eventually die out at some time.
It is shown in \cite{deta} that
the extinction probability is the minimal nonnegative solution of the following Quadratic Vector Equation (QVE):
\begin{equation}\label{yuan}
x=a+B(x\otimes x),
\end{equation}
where $x=[x_i]\in\mathbb{R}^n$ is the unknown vector,
with $x_{i}$ the probability of a population starting from an individual in state $i$ becomes extinct,
$a=[a_i]\in \mathbb{R}^n$ is the coefficient vector of QVE \eqref{yuan}
with $a_i$ the probability of an individual in state $i$ dies out without producing offspring,
$B=[b_{i, n(j-1)+k}]\in \mathbb{R}^{n\times n^2}$ is the coefficient matrix of QVE \eqref{yuan}
with $b_{i, n(j-1)+k}$ the probability that an individual in phase $i$ eventually produces a child in state $j$
and the parent switches to phase $k$ after the birth.
The symbol $\otimes$ denotes the Kronecker product.
It is easy to see that all entries of $x$, $a$ and $B$ are between $0$ and $1$.
Furthermore,  using the fact that the probabilities of all possible outcomes for an individual in state $i$ must sum to $1$,
for $i = 1, 2, \cdots, n$, we know that the vector $e$ must be a solution of \eqref{yuan}, i.e.,
\begin{equation}\label{rest}
e=a+B(e\otimes e).
\end{equation}

Let
\begin{equation}\label{def:R}
R=B(I\otimes e+e\otimes I).
\end{equation}
An MBT is called {\em subcritical, supercritical, or critical}
if $\rho(R)$ is strictly less than $1$, strictly greater than $1$, or equal to $1$, respectively \cite{athreya}.
In the subcritical and critical cases, the minimal nonnegative solution of \eqref{yuan} is the vector of all ones $e$,
while in the supercritical case, the minimal nonnegative solution $x^{\ast}$ satisfies $x^{\ast}\leq e, x^{\ast}\neq e$ \cite{detb}.
If $R$ in \eqref{def:R} is {\em irreducible}, then the MBT is {\em positive regular}.
In such case, either $x^{\ast} = 0$, or $x^{\ast} = e$, or $0 < x^{\ast} < e$.
When $R$ is reducible, Bini {\rm et. al}\cite{nlaa} show how to reduce the QVE into several smaller size QVEs
whose associated matrix $R$ is irreducible.
Hereafter, we will only concentrate in the supercritical positive regular MBT,
whose extinction probability -- the minimal nonnegative solution $x^{\ast}$ of the corresponding QVE,
is strictly less than $e$, i.e., $0\le x^{\ast}<e$.

Efforts have been devoted to the computation of the minimal nonnegative solution of QVE \eqref{yuan}.
For example, the depth and order algorithm \cite{deta}, the thickness algorithm \cite{detb},
which are all linearly convergent; the Newton's iteration \cite{newton},
which is shown to be quadratically convergent for any initial $0\le x_0\le a$;
the Perron vector iteration \cite{perron,nlaa}, which converges faster than Newton's iteration for close to critical problems
($\rho(R)$ close to 1 from the above).
All the above iterations enjoy probabilistic interpretations.
We refer the readers to \cite{poloni} and reference therein for more numerical methods.


In this paper, we study the perturbation analysis of the extinction probability of the supercritical positive regular MBT.
We give a perturbation bound and an error bound for extinction probability.
Numerical tests show that these bounds are fairly sharp.
The rest of this paper is organized as follows.
In Section~\ref{sec:bound},  a perturbation bound and an error bound for the extinction probability of the supercritical positive regular MBT are derived.
Numerical examples are presented in Section~\ref{sec:numer},
and conclusion remarks are given in Section~\ref{sec:conclusion}.

\section{The perturbation bound and error bound}\label{sec:bound}

In this section, we will derive a perturbation bound and an error bound for the extinction probability of the supercritical positive regular MBT.
First, we give some preliminary results, which will be used in the subsequent subsections.

\subsection{Preliminary}

We give three lemmas, the first two are about the supercritical positive regular MBT,
the last one is a fixed point theorem.
{\em Hereafter, if an MBT is Supercritical and Positive Regular, we will call it an SPR-MBT}.

\begin{lemma}\label{le1}\cite{athreya}
Let  $\mathcal{E}_n=\{x\in\mathbb{R} \, : \, 0 \leq x < e \}$.
If the MBT is an SPR-MBT,
the only solution of QVE \eqref{yuan} in $\mathcal{E}_n$ is the extinction probability of the MBT,
which is the minimal nonnegative solution $x^{\ast}$ of the equation \eqref{yuan}.
\end{lemma}

The proof is omitted and we refer the reader to \cite[Section V.3]{athreya} for the details of Lemma~\ref{le1}.
The following property is also given without proof as it is stated in \cite{detb,newton}.

\begin{lemma}\label{le2}\cite{detb,newton}
Let $x^{\ast}$ be the minimal nonnegative solution to QVE \eqref{yuan} and assume the associated MBT is an SPR-MBT.
Then
\begin{equation*}
    \rho(B(x^{\ast}\otimes I + I \otimes x^{\ast}))<1,
\end{equation*}
and $(I-B(x^{\ast}\otimes I + I \otimes x^{\ast}))^{-1}$ exists and is nonnegative.
\end{lemma}
The so called Leray-Schauder fixed point theorem is stated as follows.
\begin{lemma}\cite{function}
If $K$ is a nonempty compact convex set in a locally convex space $X$, and $f:K\rightarrow K$ is continuous, then $f(p)=p$ for some $p\in K$.
\end{lemma}

\bigskip
To proceed our discussions, we summarize the properties of QVE \eqref{yuan} corresponding with MBT and SPR-MBT.
The coefficient vector $a$ and matrix $B$ of QVE \eqref{yuan} associated with an MBT satisfy
\begin{enumerate}
\item[P1.] the entries of $a$ and $B$ are all between $0$ and $1$;
\item[P2.] $e=a+B(e\otimes e)$.
\end{enumerate}
The coefficient vector $a$ and matrix $B$ of QVE \eqref{yuan} associated with an SPR-MBT satisfy P1, P2, and
\begin{enumerate}
\item[P3.] the matrix $R=B(I\otimes e+e\otimes I)$ is nonnegative irreducible matrix and $\rho(R)>1$.
\end{enumerate}

\subsection{A perturbation bound}

Now let's make a small perturbation to an SPR-MBT such that the perturbed MBT is still an SPR-MBT.
And the QVE associated with  the perturbed SPR-MBT is
\begin{equation}\label{perequation}
  \tilde{x}=\tilde{a}+\tilde{B}(\tilde{x}\otimes \tilde{x}).
\end{equation}
It is easy to see that $\tilde{a}\ge 0$, $\widetilde{B}\ge 0$ and
\begin{align}\label{perturequation}
e=\tilde{a}+\widetilde{B}(e\otimes e).
\end{align}
In what follows, we shall derive a perturbation bound for the extinction probability by the following two steps:
\begin{description}
\item[Step 1.] Prove that the perturbed QVE \eqref{perequation} has a solution $\tilde{x}$, and derive a bound for $\|\tilde{x}-x^{\ast}\|$;
\item[Step 2.] Prove that $\tilde{x}$ is the minimal nonnegative solution $\tilde{x}^{\ast}$ of the perturbed QVE \eqref{perequation}.
\end{description}

\bigskip

Let $\Delta B=\widetilde{B}-B$, for the minimal nonnegative solution  $x^{\ast}$ of QVE \eqref{yuan} and $\Delta x\in\mathbb{R}^n$, define
\begin{align}
L&=I-B(x^{\ast}\otimes I+I\otimes x^{\ast}), \label{eq:L}\\
\varphi (\Delta x)&=B(\Delta x\otimes \Delta x)+\Delta B(\Delta x\otimes x^{\ast}+x^{\ast} \otimes \Delta x)+\Delta B(\Delta x\otimes \Delta x),
\label{eq:phi}\\
\mu (\Delta x)&=L^{-1}(\Delta B(x^{\ast}\otimes x^{\ast}-e\otimes e)+\varphi (\Delta x)).\label{eq:mu}
\end{align}
Step 1 can be done by using Leray-Schauder fixed point theorem.

\begin{lemma} \label{yinli1}
Let $x^{\ast}$ be the extinction probability of an SPR-MBT.
Assume that the SPR-MBT is still an MBT after perturbation.
Let the perturbed QVE \eqref{perequation} be the QVE associated with the perturbed MBT.
Denote $\Delta a=\tilde{a}-a$, $\Delta B= \widetilde{B}-B$, and let
\begin{equation}\label{lbdd}
 \delta=\| \Delta B \|,\quad   \ell=\|L^{-1}\|, \quad \tilde{b}=\|B+\Delta B\|,\quad d=\|x^{\ast}\otimes x^{\ast}-e\otimes e\|,
\end{equation}
where $L$ is defined in \eqref{eq:L}.
If $\delta$ satisfies
\begin{equation}\label{ineq:del1}
    \|x^{\ast}\|\delta +\sqrt{\tilde{b}d \delta}\le \frac{1}{2\ell},
\end{equation}
then the perturbed QVE \eqref{perequation} has a solution $\tilde{x}$ and
\begin{equation*}
    \|\tilde{x}-x^{\ast}\|\leq \xi_{\ast},
\end{equation*}
where
\begin{equation}\label{xiaogen}
\xi_{\ast}=\frac{2\ell d\delta}{1-2\ell\delta \|x^{\ast}\|+\sqrt{(1-2\ell \delta \|x^{\ast}\|)^2-4\ell^2\tilde{b}d\delta }}.
\end{equation}
\end{lemma}

\begin{proof}
Subtracting \eqref{perturequation} by \eqref{rest}, we get
\begin{equation}\label{newper}
    \Delta a=-\Delta B(e\otimes e).
\end{equation}
If the perturbed equation \eqref{perequation} has a solution $\tilde{x}=x^{\ast}+\Delta x$,
it can be rewritten as
\begin{align*}
x^{\ast}+\Delta x= a+\Delta a+ (B+\Delta B)[(x^{\ast}+\Delta x)\otimes (x^{\ast}+\Delta x)].
\end{align*}
Then using \eqref{newper} and the fact that $x^{\ast}$ satisfies \eqref{yuan}, we have
\begin{equation}\label{raodong}
    L \Delta x=\Delta B(x^{\ast}\otimes x^{\ast}-e\otimes e)+\varphi (\Delta x),
\end{equation}
where $L$ and $\varphi(\Delta x)$ is defined in \eqref{eq:L} and \eqref{eq:phi}, respectively.
In order to show \eqref{perequation} has a solution $\tilde{x}$, it suffices if we can show that \eqref{raodong} has a solution $\Delta x$.

By Lemma~\ref{le2}, we know that $L$ is nonsingular.
Consequently, the equality \eqref{raodong} can be rewritten as
\begin{equation}\label{keyeq}
\Delta x= \mu(\Delta x),
\end{equation}
where $\mu(\Delta x)$ is defined in \eqref{eq:mu}.
Then it follows
\begin{align}\label{ineq:dx}
\|\Delta x\|=\|\mu(\Delta x)\|\le \ell\delta d +2\ell \delta \|x^{\ast}\| \|\Delta x\| + \ell \tilde{b}\|\Delta x\|^2.
\end{align}

Consider a second order equation for $\xi$
\begin{equation}\label{erci}
  l\tilde{b}\xi^2 +(2l\delta \|x^{\ast}\|-1)\xi+ld\delta =0.
\end{equation}
Notice that $\delta$ satisfying \eqref{ineq:del1} implies that
$\delta$ satisfies
\begin{equation}\label{zhenggen}
    \delta\le \frac{1}{2\ell \|x^{\ast}\|}, \quad \quad \delta \leq \frac{(1-2\ell\delta \|x^{\ast}\|)^2}{4\ell^2\tilde{b} d}.
\end{equation}
Therefore, \eqref{erci} has two positive roots, and the smaller one can be given by
\begin{equation}\label{xiaogen}
\xi_{*}=\frac{1-2\ell\delta \|x^{\ast}\|-\sqrt{(1-2\ell\delta \|x^{\ast}\|)^2-4\ell^2\tilde{b}d\delta }}{2\ell\tilde{b}}
=\frac{2\ell d\delta}{1-2\ell\delta \|x^{\ast}\|+\sqrt{(1-2\ell \delta \|x^{\ast}\|)^2-4\ell^2\tilde{b}d\delta }}.
\end{equation}

Define
\begin{equation*}
    \mathcal{S}_{\xi_{*}}=\{ \Delta x\in \mathbb{R}^n\, : \, \|\Delta x\|\leq \xi_{*}\},
\end{equation*}
then it is easy to see that $\mathcal{S}_{\xi_{*}}$ is a nonempty bounded closed convex set in $\mathbb{R}^n$.
Noticing that $\mu(\Delta x)$ defined in \eqref{eq:mu} is a continuous mapping, and
for any $ \|\Delta x\|\leq \xi_{*}$,
\begin{equation*}
    \|\mu (\Delta x)\| \leq \ell d\delta  +2\ell\delta \|x^{\ast}\|\xi_{*} +\ell\tilde{b}\xi_{*}^2 =\xi_{*},
\end{equation*}
i.e., $\mu(\Delta x)$ maps  $\mathcal{S}_{\xi_{*}}$ into $\mathcal{S}_{\xi_{*}}$.
So by Leray-Schauder fixed point theorem, the function has a fixed point $\Delta x^{\ast}\in \mathcal{S}_{\xi_{*}}$,
which means \eqref{raodong} has a solution $\Delta x$.
The conclusion follows.
\end{proof}

Lemma \ref{yinli1} tells us that under some restrictions on $\delta=\|\Delta B\|$,
the perturbed equation \eqref{perequation} has a solution $\tilde{x}$, and
$\|\tilde{x}-x^{\ast}\|$ is bounded by $\xi_{\ast}$, which is a function of $\delta$.
Next we will show that under further restrictions  on $\delta$,
for any $\Delta x$ satisfying $\|\Delta x\|\le \xi_{\ast}$,
it holds $0< x^{\ast}+\Delta x < e$.
Then by Lemma~\ref{le1}, Step~2 can be accomplished.

\begin{lemma} \label{yinli2}
Follow the notations in Lemma~\ref{yinli1}, if $\delta$ satisfies
\begin{equation*}
    2\ell\delta \|x^{\ast}\|+\sqrt{(1-2\ell\delta \|x^{\ast}\|)^2-4l^2\tilde{b}d\delta }>
\max\{1-2\ell\tilde{b}( 1-\|x^{\ast}\|),1-2\ell\tilde{b}( 1-\|e-x^{\ast}\|)\},
\end{equation*}
then for any $\| \Delta x\|\le \xi_{\ast}$, it holds $0< x^{\ast}+\Delta x <  e$.
\end{lemma}

\begin{proof}
If $\delta$ satisfies
\[
2\ell\delta \|x^{\ast}\|+\sqrt{(1-2\ell\delta \|x^{\ast}\|)^2-4l^2\tilde{b}d\delta }>
\max\{1-2\ell\tilde{b}( 1-\|x^{\ast}\|),1-2\ell\tilde{b}( 1-\|e-x^{\ast}\|)\},
\]
then
\begin{align*}
\xi_{\ast}=\frac{1-2\ell\delta \|x^{\ast}\|-\sqrt{(1-2\ell\delta \|x^{\ast}\|)^2-4l^2\tilde{b}d\delta }}{2\ell\tilde{b}}
< {\min\{ 1-\|x^{\ast}\|, 1-\|e-x^{\ast}\|\}}.
\end{align*}
Consequently, for any $\|\Delta x\| \le \xi_{\ast}$, and for all $i=1,2,\dots, n$, we have
\begin{align*}
x^{\ast}_i+\Delta x_i &\le \|x^{\ast}\|+\xi_{\ast}< 1,\\
x^{\ast}_i-\Delta x_i & \ge (1-\|e-x^{\ast}\|) -\xi_{\ast}> 0,
\end{align*}
where $x_i^{\ast}$ and $\Delta x_i$ are the $i$-th entry of $x^{\ast}$ and $\Delta x$, respectively.
The conclusion follows immediately.
\end{proof}

Combining Lemma~\ref{yinli1} and Lemma~\ref{yinli2}, we get the following theorem,
which gives the perturbation bound for the extinction probability of the SPR-MBT.

\begin{theorem} \label{theo2}
Let $x^{\ast}$ be the extinction probability of an SPR-MBT.
Assume that the SPR-MBT is still an SPR-MBT after perturbation.
Let the perturbed QVE \eqref{perequation} be the QVE associated with the perturbed SPR-MBT.
Let $\Delta a$, $\Delta B$,  $\delta$, $\ell$, $\tilde{b}$ and $d$ be defined as in Lemma~\ref{yinli1}.
If $\delta$ satisfies
\begin{align*}
   \|x^{\ast}\|\delta +\sqrt{\tilde{b} d \delta}&\le\frac{1}{2\ell},\\
2\ell\delta \|x^{\ast}\|+\sqrt{(1-2\ell\delta \|x^{\ast}\|)^2-4l^2\tilde{b}d\delta } &>
\max\{1-2\ell\tilde{b}( 1-\|x^{\ast}\|),1-2\ell\tilde{b}( 1-\|e-x^{\ast}\|)\},
\end{align*}
then the extinction probability $\tilde{x}^{\ast}$ of the perturbed SPR-MBT
satisfies
\begin{equation*}
    \|\tilde{x}^{\ast}-x^{\ast}\|\leq \xi_{\ast},
\end{equation*}
where $\xi_{\ast}$ is defined in \eqref{xiaogen}.
\end{theorem}

\begin{proof}
By Lemma~\ref{yinli1}, we know that the perturbed equation \eqref{perequation} has a solution $\tilde{x}$,
which satisfies
\[
\|\tilde{x}-x^{\ast}\|\le \xi_{\ast}.
\]
For such $\tilde{x}$, using Lemma~\ref{yinli2}, we have $0< \tilde{x}< e$.
Then it follows from Lemma~\ref{le1} that $\tilde{x}$ is the minimal nonnegative solution $\tilde{x}^{\ast}$ of the perturbed QVE \eqref{perequation},
the conclusion follows immediately.
\end{proof}

Several remarks follow.
\begin{remark}{\rm~
\begin{enumerate}
\item
Theorem~\ref{theo2}  tells that when an SPR-MBT is slightly perturbed
such that the perturbed MBT is still an SPR-MBT,
the extinction probabilities of two MBTs are close,
and can be bounded by $\xi_{\ast}$, a function of the perturbation $\delta=\|\Delta B\|$.

\item
When $\delta$ is sufficiently small, direct calculation leads to
\[
\xi_{*}= \ell d\delta +O(\delta^2).
\]
By Theorem~\ref{theo2}, we can obtain the first order absolute perturbation bound for  $x^{\ast}$
\begin{equation}\label{yijiejuedui}
    \|\tilde{x}^{\ast}-x^{\ast}\|\leq \ell d \|\Delta B\|+O(\|\Delta B\|^2).
\end{equation}
and the first order relative perturbation bound for  $x^{\ast}$
\begin{equation}\label{yijiexiangdui}
    \frac{\|\tilde{x}-x^{\ast}\|}{\|x^{\ast}\|}\leq \ell \|B\| \frac{d}{\|x^{\ast}\|} \frac{\|\Delta B\|}{\|B\|}+O((\frac{\|\Delta B\|}{\|B\|})^2).
\end{equation}

According to the theory of condition developed by Rice \cite{rice}, we
define the condition number $\kappa$ of $x$ by
\begin{align}\label{tiaojianshu}
\kappa=\lim_{\epsilon \to \infty}\mbox{sup}\{\frac{\|\Delta x\|}{\epsilon \|x\|}\, : \,
x+\Delta x&=a+\Delta a+(B+\Delta B)((x+\Delta x)\otimes (x+\Delta x)),\notag\\
e&=a+\Delta a+(B+\Delta B)(e\otimes e),\|\Delta B\|\leq \epsilon\|B\|\}.
\end{align}
Then it follows that the condition number of $x^{\ast}$ can be bounded by
\begin{align*}
\kappa\le \frac{\|\tilde{x}^{\ast}-x^{\ast}\|}{\|x^{\ast}\|}\frac{\|B\|}{\delta}
\le  \frac{\ell d \|B\|}{\|x^{\ast}\|}+O(\delta).
\end{align*}
Therefore, we may use
\begin{equation}\label{kappastar}
    \tilde{\kappa}=\frac{\ell d \| B\|}{\|x^{\ast}\|}
\end{equation}
to measure the sensitivity of $x^{\ast}$.

\item For the extinction probability $x^{\ast}$ of an SPR-MBT, it satisfies \eqref{yuan}. i.e.,
\[
x^{\ast}=a+B(x^{\ast}\otimes x^{\ast}).
\]
Subtracting it from \eqref{rest}, we can obtain
\[
B(e\otimes e-x^{\ast}\otimes x^{\ast})=e-x^{\ast},
\]
which can be rewritten as
\[
B[(e+x^{\ast})\otimes I +I\otimes (e+x^{\ast})](e-x^{\ast})=2(e-x^{\ast}).
\]
Notice that $B[(e+x^{\ast})\otimes I +I\otimes (e+x^{\ast})]$ is nonnegative irreducible matrix, and $e-x^{\ast}>0$,
then it follows from the Perron-Frobenius theorem that
\[
\rho(B[(e+x^{\ast})\otimes I +I\otimes (e+x^{\ast})])=2.
\]
Notice also that $B[(e+x^{\ast})\otimes I +I\otimes (e+x^{\ast})]=R+(I-L)$, where $R$ and $L$ are defined in \eqref{def:R} and \eqref{eq:L}, respectively,
and both of them are nonnegative matrices,
then
\[
\rho(R)+\rho(I-L) \ge \rho(B[(e+x^{\ast})\otimes I +I\otimes (e+x^{\ast})])=2.
\]
Therefore, for close to critical SPR-MBT, as $\rho(R)$ goes to $1$ from the above,
$\rho(I-L)$ must go to $1$ from the below.
As a result, $L$ tends to be singular and $\ell=\|L^{-1}\|$ goes to infinity.
However, this does not necessarily lead to a large condition number,
since $\tilde{\kappa}$ is not only related with $\ell=\|L^{-1}\|$, but also $d=\|x^{\ast}\otimes x^{\ast}- e\otimes e\|$.
Numerically, we found that
the larger $\ell$ is, the smaller $d$ becomes.
As a result, $\tilde{\kappa}$ is relative small even if $\ell$ is large,
which means the computation of the extinction probability of close to critical SPR-MBT is NOT  very ill-conditioned as we thought.
This numerical phenomenon is interesting, however, we failed to give a theoretical proof.

\end{enumerate}
}\end{remark}

\subsection{An error bound}\label{subsec:err}
In this section, we provide an error bound for the extinction probability of the SPR-MBT,
which can be used measure the quality of an approximate solution of QVE \eqref{yuan}.

Let $\hat{x}$ be an approximate solution of QVE \eqref{yuan}, and denote the residual by
 \begin{equation}\label{jisuan}
 r=\hat{x}-a-B(\hat{x}\otimes \hat{x}).
\end{equation}
As an approximate extinction probability of an SPR-MBT,
it is natural to impose the following conditions on $\hat{x}$:
\begin{equation}\label{con1}
   0\le \hat{x}<e, \quad  \rho(B(\hat{x}\otimes I + I \otimes \hat{x}))<1.
\end{equation}
Notice that if $\hat{x}$ obtained via the depth, order, thickness algorithms or Newton iteration,
then $a\le \hat{x}\le x^{\ast}$, and hence \eqref{con1} holds since
$a\ge 0$, $x^{\ast}<e$, and $ \rho(B(\hat{x}\otimes I + I \otimes \hat{x}))\le  \rho(B({x}^{\ast}\otimes I + I \otimes {x}^{\ast}))<1$.

\bigskip

Now we give the error bound for the extinction probability of the SPR-MBT.
\begin{theorem}\label{theo3}
Let $x^{\ast}$ be the minimal nonnegative solution to QVE \eqref{yuan},
which associated with an  SPR-MBT.
For an approximate solution $\hat{x}$ of QVE \eqref{yuan},
denote the residual vector as in \eqref{jisuan} and let $\gamma=\|r\|$.
If $\hat{x}$ satisfies \eqref{con1} and $\gamma$ satisfies
\begin{align}
 1-4\hat{\ell}^2b\gamma & \geq  0 ,\label{con21}\\
 \sqrt{1-4\hat{\ell}^2b\gamma} & > \max\{1-2\hat{\ell}b(1-\|e-\hat{x}\|),1-2\hat{\ell}b(1-\|\hat{x}\|)\},\label{con22}
\end{align}
where $\widehat{L}=I-B(\hat{x}\otimes I + I \otimes \hat{x})$, $\hat{\ell}=\|\widehat{L}^{-1}\|$ and $b=\|B\|$,
then
\begin{equation*}
    \|\hat{x}-x^{\ast}\|\leq \frac{2\hat{\ell}\gamma}{1+\sqrt{1-4\hat{\ell}^2b\gamma}}:=\omega_{\ast}.
\end{equation*}
\end{theorem}

\begin{proof}
Subtracting (\ref{yuan}) from (\ref{jisuan}) , we get
\begin{equation}\label{dxr}
    \hat{x}-x^{\ast}-B((\hat{x}-x^{\ast})\otimes \hat{x}+x^{\ast}\otimes(\hat{x}-x^{\ast}))=r.
\end{equation}
Let $\Delta x=\hat{x}-x^{\ast}$, then \eqref{dxr} can be rewritten as
\begin{equation*}
 \widehat{L}\Delta x=   [I-B(\hat{x}\otimes I + I \otimes \hat{x})]\Delta x=r-B(\Delta x \otimes \Delta x).
\end{equation*}
Using the assumption that $\rho(B(\hat{x}\otimes I + I \otimes \hat{x}))<1$, we know $\widehat{L}$ is nonsingular.
Consequently,
\[
\Delta x=\widehat{L}^{-1}(r-B(\Delta x \otimes \Delta x)),
\]
which leads to
\begin{equation*}
\|\Delta x\|\leq \|\widehat{L}^{-1}\|\|r\|+\|\widehat{L}^{-1}\|\|B\|\|\Delta x\|^2=\hat{\ell}\gamma+\hat{\ell}b\|\Delta x\|^2.
\end{equation*}
Define
\begin{equation*}
\hat{\mu}(\Delta x)=\widehat{L}^{-1}(r-B(\Delta x \otimes \Delta x).
\end{equation*}
Similar as in the proof of Lemma~\ref{yinli1}, we can show that if
\begin{equation*}
    1-4\hat{\ell}^2b\gamma \geq 0,
\end{equation*}
then for any $\|\Delta x\|\leq \omega_{*}$, it holds $\|\hat{\mu}(\Delta x)\|\le \|\omega_{\ast}$,
where
\begin{equation}\label{omega}
   \omega_*= \frac{1-\sqrt{1-4\hat{\ell}^2b\gamma}}{2\hat{\ell}b}=\frac{2\hat{\ell}\gamma}{1+\sqrt{1-4\hat{\ell}^2b\gamma}}
\end{equation}
is the smaller positive root of the second order equation
\begin{equation}\label{jisuanerci}
    \hat{\ell}b\omega^2-\omega +\hat{\ell}\gamma=0.
\end{equation}
So by Leray-Schauder fixed-point theorem again, the continous function $\hat{\mu}(\Delta x)$ has a fixed point,
in the bounded closed convex set $\mathcal{S}_{\omega_{*}} \subset\mathbb{R}^n$ defined by
\begin{equation*}
    \mathcal{S}_{\omega_{*}}=\{ \Delta x\in \mathbb{R}^n \, : \, \|\Delta x\|\leq \omega_{*}\}.
\end{equation*}
In another word, there exists a $\Delta x^{\ast}\in \mathcal{S}_{\omega_{*}}$ such that
\begin{equation*}
\Delta x^{\ast}=\widehat{L}^{-1}(r-B(\Delta x^{\ast} \otimes \Delta x^{\ast}).
\end{equation*}
Combing this equation with equation (\ref{jisuan}), we can get that $x=\hat{x}-\Delta x^{\ast}$ satisfies
\begin{equation*}
    x-a-B(x\otimes x)=0.
\end{equation*}

According to Lemma \ref{le1}, if we can show that $0\le \hat{x}-\Delta x^{\ast}<e$,
then $\hat{x}-\Delta x^{\ast}$ is the minimal nonnegative solution of QVE (\ref{yuan}).
Similar as the proof of Lemma~\ref{yinli2}, using \eqref{con22}, we can show $0< \hat{x}-\Delta x^{\ast}<e$,
which completes the proof.
\end{proof}

\begin{remark}{\rm~
\begin{enumerate}
\item
The error bound $\omega_{\ast}=\omega_{\ast}(\gamma)$ is a function of the residual norm $\gamma$.
A fundamental requirement for the sharpness of the error bound is that $\omega_{\ast}(0)=0$,
which is satisfied by $\omega_{\ast}$ given in Theorem~\ref{theo3}.

\item
Direct calculation show that
\[
\omega_{\ast}=\hat{\ell}\gamma+O(\gamma^2).
\]
Consequently, we may use
\begin{equation*}
    \hat{\ell}\gamma=\|[I-B(\hat{x}\otimes I + I \otimes \hat{x})]^{-1}\|\|r\|
\end{equation*}
to estimate the error $\|\hat{x}-x^{\ast}\|$ when $\gamma$ is sufficiently small.
\end{enumerate}
}\end{remark}

\section{Numerical Experiments}\label{sec:numer}

In this section,
we illustrate our perturbation bound and error bound by a few numerical examples.
Using MATLAB's Symbolic Math Toolbox,
we adopt the Newton's method \cite{newton} with the stopping criterion
  \begin{equation*}
    \|a+B(x\otimes x)-x\|_2<10^{-50}
\end{equation*}
to evaluate the ``exact'' solution $x^{\ast}$ of QVE \eqref{yuan}.
Similarly, we evaluate $\tilde{x}^{\ast}$ from QVE \eqref{perequation}.

\paragraph{QVE Data}
We take the test data from  \cite{detb}.
The coefficient vector $a$ and matrix $B$ in QVE \eqref{yuan}
can be given by
\begin{align*}
a=-D_0^{-1}d,\qquad B=-D_0^{-1}R,
\end{align*}
where
\begin{align*}
   D_0=10^{-3}\left[\begin{smallmatrix}
  \cdot & 6 & 0 & 0 & 0 & 0 & 0 & 0 & 0 \\
  0 & \cdot & 6 & 0 & 0 & 0 & 0 & 0 & 0 \\
  0 & 0 & \cdot & 6 & 0 & 0 & 0 & 0 & 0 \\
  6 & 0 & 0 & \cdot & 1 & 1 & 0 & 0 & 0 \\
  0 & 0 & 0 & 0 & \cdot & 0 & 0 & 0 & 0 \\
  0 & 0 & 0 & 0 & 0 & \cdot & 6 & 0 & 0 \\
  0 & 0 & 0 & 0 & 0 & 0 & \cdot & 6 & 0 \\
  0 & 0 & 0 & 0 & 0 & 0 & 0 & \cdot & 6 \\
  1 & 0 & 0 & 0 & 1 & 6 & 0 & 0 & \cdot
\end{smallmatrix}\right],\quad
d=\diag(0,0,0,0,1,0,0,0,0),\quad
R=[R_{i,jk}]=[(D_1)_{ii}(P_1)_{ij}(P_0)_{ik}]
\end{align*}
with
\begin{align*}
D_1=10^{-2}\diag( p,  p ,  p ,  p,   5 ,   4,  4 ,  4 ,  4 ),\quad
    P_1=\left[\begin{smallmatrix}
  1 & 0 & 0 & 0 & 0 & 0 & 0 & 0 & 0 \\
  0 & 1 & 0 & 0 & 0 & 0 & 0 & 0 & 0 \\
  0 & 0 & 1 & 0 & 0 & 0 & 0 & 0 & 0 \\
  0 & 0 & 0 & 1 & 0 & 0 & 0 & 0 & 0 \\
  0.1 & 0 & 0 & 0 & 0.9 & 0 & 0 & 0 & 0 \\
  0 & 0 & 0 & 0 & 1 & 0 & 0& 0 & 0 \\
  0 & 0 & 0 & 0 & 1 & 0 & 0 & 0 & 0 \\
  0 & 0 & 0 & 0 & 1 & 0 & 0 & 0 & 0 \\
  1 & 0 & 0 & 0 & 1 & 0 & 0 & 0 & 0
\end{smallmatrix}\right],\quad
 P_0=\left[\begin{smallmatrix}
  0 & 0 & 0 & 0 & 1 & 0 & 0 & 0 & 0 \\
  0 & 0 & 0& 0 & 1 & 0 & 0 & 0 & 0 \\
  0 & 0 & 0 &0 & 1& 0 & 0 & 0 & 0 \\
  0 & 0 & 0 & 0 & 1 & 0 & 0 & 0 & 0 \\
  0.1 & 0 & 0 & 0 & 0.9 & 0 & 0 & 0 & 0 \\
  0 & 0 & 0 & 0 & 0 & 1 & 0 & 0 & 0 \\
  0 & 0 & 0 & 0 & 0 & 0 & 1 & 0 & 0 \\
  0 & 0 & 0 & 0 & 0 & 0 & 0 & 1 & 0 \\
  0 & 0 & 0 & 0 & 0 & 0 & 0 & 0 & 1
\end{smallmatrix}\right],
\end{align*}
and the diagonal of $D_0$ is such that $D_0e + D_1e + d = 0$, $p$ is a real parameter.

\begin{example}
In this example,
we add two kinds of perturbations to QVE \eqref{yuan} to illustrate the sharpness of our perturbation bound.
The first one is the structured perturbation given below:
\begin{equation*}
    \Delta B=\eta  B ,\quad  \Delta a=-\Delta B(e\otimes e),
\end{equation*}
and the second one is the randomly generated perturbation:
\begin{equation*}
    \Delta B=\rand(n,n^2),\quad \Delta B=\eta \frac{\Delta B}{\|\Delta B\|}\|B\|, \quad \Delta a=-\Delta B(e\otimes e),
\end{equation*}
where  $\rand(n,n^2)$ is a $n$-by-$n^2$ matrix whose elements are drawn from the standard uniform distribution on the open interval $(0,1)$.

For different real parameters $p$ and $\eta$,
we list the numerical results for both cases in Table~\ref{tab1} and Table~\ref{tab2}, respectively,
where $\rho(R)$ is the spectral radius of $R$ defined in \eqref{def:R}, $\ell$, $d$ are defined in \eqref{lbdd},
$\tilde{\kappa}$ is the upper bound for the condition number of $x^{\ast}$ defined in \eqref{kappastar},
and  $\eta=\frac{\|\Delta B\|}{\|B\|}$ is the relative perturbation in $B$.

\begin{table}[h]
\centering
\begin{tabular}{|c|c|c|c|c|c|c|c|}
                        \hline
                         $p$  &$\rho(R)$ & $\ell$&$d$&$\tilde{\kappa}$&$\eta$ & $\frac{\xi_{\ast}}{\|x^{\ast}\|}$ & $\frac{\|\tilde{x}^{\ast}-x^{\ast}\|}{\|x^{\ast}\|}$  \\
\hline\hline
                         20.0&1.0095&1,29e+2&0.99&1.29e+2&1.0e-8&1.30e-6&1.02e-6\\
                         &  & & & &1.0e-9 & 1.29e-7 & 1.02e-7 \\
\hline
                         10.0 &1.0084&1.64e+2&0.96 & 1.59e+2&1.0e-8 & 1.59e-6 & 1.01e-6 \\
                          &  & & & &1.0e-9 & 1.59e-7 & 1.01e-7 \\
\hline
                         5.0 &1.0065&2.44e+2&0.85&2.09e+2& 1.0e-8 & 2.09e-6 &  9.85e-7\\
                         & & & && 1.0e-9 & 2.09e-7 & 9.85e-8 \\
\hline
                           2.0 &1.0028&7.36e+2&0.54&3.95e+2& 1.0e-8 & 3.96e-6 &  1.16e-6\\
                         & & & && 1.0e-9 & 3.95e-7 & 1.16e-7 \\
\hline
                         0.90 &1.0001&1.78e+4&3.86e-2&6.85e+2& 1.0e-8 & 7.99e-6 &  1.43e-6\\
                         & & & && 1.0e-9 & 6.94e-7 & 1.43e-7 \\
                         \hline
                       \end{tabular}
\caption{Numerical results for the structured perturbations}\label{tab1}
\end{table}

\begin{table}[h]
\centering
\begin{tabular}{|c|c|c|c|c|c|c|c|}
                        \hline
                        $p$  &$\rho(R)$ & $\ell$&$d$&$\tilde{\kappa}$&$\eta$ & $\frac{\xi_{\ast}}{\|x^{\ast}\|}$ & $\frac{\|\tilde{x}^{\ast}-x^{\ast}\|}{\|x^{\ast}\|}$  \\
\hline\hline
                         20.0&1.0095&1,29e+2&0.99&1.29e+2&1.0e-8&1.30e-6&1.06e-6\\
                         &  & & & &1.0e-9 & 1.29e-7& 1.10e-7  \\
\hline
                        10.0 &1.0084&1.64e+2&0.96 & 1.59e+2&1.0e-8 & 1.59e-6 & 1.32e-6 \\
                          &  & & & &1.0e-9 & 1.59e-7 & 1.27e-7 \\
\hline
                         5.0 &1.0065&2.44e+2&0.85&2.09e+2& 1.0e-8 & 2.09e-6 & 1.60e-6\\
                         & & & && 1.0e-9 & 2.09e-7 & 1.74e-7 \\
\hline
                           2.0 &1.0028&7.36e+2&0.54&3.95e+2& 1.0e-8 & 3.96e-6 & 2.89e-6\\
                         & & & && 1.0e-9 & 3.95e-7 & 2.71e-7 \\
\hline
                         0.90 &1.0001&1.78e+4&3.86e-2&6.85e+2& 1.0e-8 & 7.99e-6 & 4.40e-6\\
                         & & & && 1.0e-9 & 6.94e-7 & 4.67e-7 \\
                         \hline
                       \end{tabular}
\caption{Numerical results for the randomly generated perturbations}\label{tab2}
\end{table}

We can see from Tables~\ref{tab1} and \ref{tab2} that as $p$ decreases, the spectral radius of $R$ get closer to $1$ from the above,
meanwhile, $\ell$ becomes larger and $d$ becomes smaller,
what's more important, the upper bound for the condition number $\tilde{\kappa}$ though becomes larger, remains the same order of magnitude.
In all cases, regardless of the kinds of perturbations, $p$ and $\eta$,
our perturbation bound  $\frac{\xi_{\ast}}{\|x^{\ast}\|}$ is in the same order of  $\frac{\|\tilde{x}^{\ast}-x^{\ast}\|}{\|x^{\ast}\|}$.

\end{example}

\begin{example}

W use this example to demonstrate that error bound $\omega_{*}$ is fairly good for estimating $\|\hat{x}-x^{\ast}\|$,
where $\hat{x}$ is an approximate solution of QVE \eqref{yuan} obtained via Newton's iteration\cite{newton},
$x^{\ast}$ is minimal nonnegative solution of QVE \eqref{yuan}.

For different $p$, we list the numerical results in Table~\ref{tab3}.
\begin{table}[h]
\centering
\begin{tabular}{|c|c|c|c|c|}
                        \hline
                         $p$  &$\|r\|$ & $\|\widehat{L}^{-1}\|$& $\omega_{*}$ & $\|\hat{x}-x^{\ast}\|$   \\\hline
                         2.0&  4.9e-9& 736 & 3.60e-6&1.83e-6\\
                         4.0 &  3.5e-9&289& 1.01e-6 & 6.51e-7 \\
                         6.0&  3.3e-7&215& 7.22e-5 & 4.93e-5\\
                         8.0 &5.0e-8& 182 & 9.17e-6 & 6.75e-6 \\
                         10.0 &8.3e-6& 163 & 2.03e-3 &  1.02e-3\\
                         \hline
\end{tabular}
\caption{Comparison between the error bound $\omega_{*}$ and the error $\|\hat{x}^{\ast}-x^{\ast}\|$ }\label{tab3}
\end{table}
We can see from Table~\ref{tab3} that the quality of an approximation $\hat{x}$ can be well measured by $\omega_{*}$ and
 we can roughly estimate the error $\|\hat{x}^{\ast}-x^{\ast}\|$ by $\|\widehat{L}^{-1}\|\|r\|$,
where $\widehat{L}$ is defined in Theorem~\ref{theo3}.


\end{example}

\section{Conclusions}\label{sec:conclusion}
In this paper, we study the perturbation analysis for the extinction probability of an SPR-MBT.
We derive a perturbation bound for the extinction probability of an SPR-MBT,
and  a upper bound for the condition number of the extinction probability are also obtained.
Furthermore, an error bound is given for the extinction probability,
which can be used to measure the quality of the approximate extinction probability.
Numerical tests show that both the perturbation bound and the error bound are fairly sharp.

\end{document}